\documentclass[12pt,reqno]{amsart}
\usepackage[T1]{fontenc}
\usepackage{lmodern}
\usepackage{amsmath,amssymb,amsthm,mathtools,enumitem,array}
\usepackage{microtype}
\usepackage[margin=1.0in]{geometry}

\usepackage[table]{xcolor}
\usepackage{booktabs}
\usepackage[colorlinks=true,linkcolor=teal!85!black,citecolor=teal!85!black,urlcolor=teal!85!black]{hyperref}
\usepackage{xcolor}
\usepackage[alphabetic]{amsrefs}
\allowdisplaybreaks
\mathtoolsset{showonlyrefs,showmanualtags} 

\newtheorem{theorem}{Theorem}[section]
\newtheorem{proposition}[theorem]{Proposition}
\newtheorem{lemma}[theorem]{Lemma}

\theoremstyle{definition}
\newtheorem{definition}[theorem]{Definition}
\newtheorem{question}[theorem]{Question}
\theoremstyle{remark}
\newtheorem{remark}[theorem]{Remark}

\numberwithin{equation}{section}
\newcommand{\T}{\mathbb T}
\newcommand{\N}{\mathbb N}
\newcommand{\Z}{\mathbb Z}
\newcommand{\R}{\mathbb R}

\newcommand{\bs}{\boldsymbol}

\makeatletter
\def\l@section{\@tocline{1}{0pt}{0pt}{2pc}{\bfseries}}
\def\l@subsection{\@tocline{2}{0pt}{2pc}{0pt}{}}
\makeatother

\title[Exponential decay of threshold correlations]{Exponential decay of threshold correlations for symmetric periodic functions}
\author[Giannitsi]{Christina Giannitsi}
    \address{Department of Mathematics, Virginia Tech, Blacksburg, VA, USA}
    \email {cgiannitsi@vt.edu}
\author[Madrid]{Jose Madrid}
    \address{Department of Mathematics, Virginia Tech, Blacksburg, VA, USA}
    \email {josemadrid@vt.edu}
    

\begin{document}

\begin{abstract}
We study how often $n$ distinct integer dilates of a periodic function can all lie above a fixed threshold or all lie below its negative. We consider continuous functions $f:\R\to[-1,1]$ that are $1$-periodic, even, strictly decreasing on $[0,1/2]$, and satisfy $f(0)=1$ and $f(t+1/2)=-f(t)$. This class includes the cosine and the triangle wave. For $0\leq\alpha<1$, let $\rho_\alpha\in(0,1/4]$ be determined by $f(\rho_\alpha)=\alpha$, and let $F_\alpha(n)$ denote the infimum of these proportions over all sets of $n$ distinct positive integer frequencies.

We prove that $F_\alpha(n)\geq 4\rho_\alpha^n$ and that the limit $\lim_{n\to\infty}F_\alpha(n)^{1/n}$ exists for every threshold. When $\rho_\alpha^{-1}$ is an integer, we obtain
\begin{align}
F_\alpha(n)
&\leq C_\alpha n^{\log_2((1-\rho_\alpha)/(1-2\rho_\alpha))}\rho_\alpha^n,
\end{align}
and hence the limit equals $\rho_\alpha$. Otherwise, the limit belongs to the interval $[\rho_\alpha,\lfloor\rho_\alpha^{-1}\rfloor^{-1}]$.

The upper bound in the reciprocal case uses a cyclic covering construction previously studied in the context of circular graphs. We give a self-contained proof of the corresponding optimal modulus.

\end{abstract}

\maketitle
\tableofcontents

\section{Introduction}
Let $A$ be a finite set of positive integers. We are interested in the proportion of points $t\in\T$ for which the functions $\cos(2\pi at)$, $a\in A$, are either all positive or all negative. We ask how small this proportion can be when the number of frequencies is fixed. Although the cosines are all positive in a neighborhood of the origin, the size of this neighborhood depends on the largest frequency and does not give a bound in terms of $|A|$ alone.

Sign correlations of this kind arose in the study of eigenfunctions of differential operators. Gon\c{c}alves, Oliveira e Silva, and Steinerberger studied a universality phenomenon for such correlations, obtaining the sharp two-frequency value $1/3$ in the cosine model \cite{GOS}. Dou, Goh, Liu, Legate, and Pettigrew subsequently proved that the sharp value for three frequencies is $1/9$ and characterized the cases of equality \cite{DGLLP}. They also observed that the powers of three, for which the common-sign proportion is $3^{1-n}$, are not always optimal. Indeed, the set $\{1,3,11,33\}$ gives a proportion of $1/33$, smaller than the value $1/27$ obtained from $\{1,3,9,27\}$. They conjectured that $1/33$ is optimal for four frequencies. 
This conjecture was recently proved by Zang, Zhao, Zhao, and Zhou, who also characterized all cases of equality \cite{ZZZZ2026}.

In this paper we study the asymptotic behavior of these correlations as the number of frequencies increases. In particular, we show that the infimum of the common-sign proportions for $n$ cosine frequencies is $4^{-n+o(n)}$. We also consider a threshold version of the problem for a class of symmetric periodic functions that includes the cosine and the triangle wave. We prove a universal lower bound and show that the exponential decay rate exists for every threshold. For an infinite family of thresholds we determine this rate exactly, while for the remaining thresholds we obtain upper and lower bounds.

\subsection{Main Results}

Let $\T = \R/\Z$, and for $t \in \R$, write $\|t\|_{\T} := \min_{k \in \Z} |t - k|$. We denote by $\tau : \T \to [-1,1]$ the triangle wave defined by $\tau(t) := 1 - 4\|t\|_{\T}.$
\begin{definition}\label{def:antipodal-wave}
Let $\mathcal{W}$ denote the class of functions $f:\R\to[-1,1]$ of the form
\begin{align}
f &= \phi\circ\tau,
\end{align}
where $\phi:[-1,1]\to[-1,1]$ is continuous, strictly increasing, odd, and satisfies $\phi(1)=1$. We refer to the functions in $\mathcal{W}$ as \emph{antipodal waves}.
\end{definition}
The condition $\phi(1)=1$ is only a normalization; more generally, one may allow any positive value of $\phi(1)$ and rescale the resulting function accordingly.

We now formulate the threshold-correlation problem for a fixed antipodal wave. Throughout, $\N=\{1,2,\ldots\}$, and $m$ denotes normalized Lebesgue measure on $\T$. Fix $f\in\mathcal{W}$. For $0\leq\alpha<1$, let $\rho_\alpha\in(0,1/4]$ be the unique number such that $f(\rho_\alpha) = \alpha.$ For a finite nonempty set $A\subset\N$, define
\begin{align}
S_\alpha(A)
&:= \{t\in\T:
f(at)>\alpha \text{ for every }a\in A,
\text{ or }
f(at)<-\alpha \text{ for every }a\in A\}.
\end{align}
For every integer $n\geq1$, let $P_\alpha (A) = m\big( S_\alpha(A)\big)$ and define
\begin{align}
F_\alpha(n) &:= \inf\{ P_\alpha(A)  :A\subset\N,\ |A|=n\}.
\end{align}

\begin{theorem}\label{thm:main}
Let $0 \leq \alpha < 1$. Then, for every integer $n \geq 1$,
\begin{align}
F_\alpha(n)
&\geq 4 \rho_\alpha^n.
\label{eq:main-lower}
\end{align}
\begin{enumerate}[leftmargin=20pt, label=\textnormal{(\roman*)}]
\item \label{pt:main-integer}
If $\rho_\alpha^{-1}\in\N$, equivalently if $\alpha = f(1/m)$ for some integer $m \geq 4$, then there is a constant $C_\alpha > 0$ such that
\begin{align}
F_\alpha(n) 
    & \leq C_\alpha 
    n^{\log_2((1-\rho_\alpha)/(1-2\rho_\alpha))}
    \rho_\alpha^n
\end{align}
for every integer $n \geq 1$. In particular,
\begin{align}
\lim_{n\to\infty}F_\alpha(n)^{1/n} &= \rho_\alpha.
\end{align}
\item \label{pt:main-noninteger}
If $\rho_\alpha^{-1}\notin\N$, then there is a constant $C_\alpha > 0$ such that
\begin{align}
F_\alpha(n) &\leq C_\alpha
    \left\lfloor \rho_\alpha^{-1}\right\rfloor^{-n}
\end{align}
for every integer $n \geq 1$. Consequently, the following limit exists and satisfies
\begin{align}
\rho_\alpha & \leq \lim_{n\to\infty}F_\alpha(n)^{1/n}
    \leq \frac{1}{\left\lfloor \rho_\alpha^{-1}\right\rfloor}.
\end{align}
\end{enumerate}
\end{theorem}

The reciprocal upper bound uses a cyclic covering construction closely related to earlier work on the Shannon capacity of circular graphs. The corresponding optimal modulus and cyclic separation construction were obtained by Polak and Schrijver \cite{Polak2019}*{Theorem 9.3.1, 9.3.3} and later recovered in the more general framework of Buys, Polak, and Zuiddam \cite{BPZ2025}. We state the result in the form needed here and include a self-contained proof.

\begin{theorem}\label{thm:main-cover}
Let $0 \leq \alpha < 1$ and suppose that $\rho_\alpha^{-1}\in\N$. For every integer $d \geq 0$, define
\begin{align}
q_d &= \frac{(1 - 2\rho_\alpha)\rho_\alpha^{-d} 
    + \rho_\alpha}
    {1 - \rho_\alpha}.
\end{align}
Also, let $D_1 = \{1\}$, and, for $d \geq 2$, 
\begin{align}
D_d &= \{q_{d-1}\} \cup \{\rho_\alpha^{-j}:0 \leq j \leq d - 2\}.
\end{align}
Then, $D_d$ consists of $d$ distinct positive integers, and for every integer $k$ with $1 \leq k < q_d$, there exists $b\in D_d$ such that
\begin{align}
\left\|\frac{bk}{q_d}\right\|_{\T} & > \rho_\alpha.
\end{align}
Moreover, $q_d$ is the largest modulus for which such a strict cyclic cover with $d$ multipliers is possible. More precisely, if $q \geq 2$ is an integer and $b_1,\ldots,b_d\in\Z$ satisfy
\begin{align}
\max_{1 \leq j \leq d}
\left\|\frac{k b_j}{q}\right\|_{\T}
&> \rho_\alpha
\qquad\text{for every }1 \leq k < q,
\end{align}
then $q \leq q_d$.
\end{theorem}

Since $\tau$ maps $[0,1/2]$ bijectively onto $[-1,1]$, the representation $f = \phi\circ\tau$ determines $\phi$ uniquely through 
\begin{align}
\phi(s) &= f\left(\frac{1-s}{4}\right),
\qquad -1\leq s\leq1.
\end{align}

\begin{proposition}\label{prop:wave-characterization}
Let $f:\R\to[-1,1]$. $f\in\mathcal{W}$ if and only if the following properties hold:
\begin{enumerate}[label=\textnormal{\textbf{\roman*.}}, leftmargin=*]
\item $f$ is $1$-periodic.
\item $f$ is even.
\item For every $t\in\R$, $\displaystyle f\left(t+\frac12\right) = -f(t).$
\item The restriction of $f$ to $[0,1/2]$ is strictly decreasing.
\item $f(0)=1$.
\item $f$ is continuous.
\end{enumerate}
Moreover, the representation $f=\phi\circ\tau$ is unique.
\end{proposition}

The class $\mathcal{W}$ contains a variety of familiar periodic profiles. Table \ref{tab:antipodal-wave-examples} lists several examples together with the corresponding change of scale $\phi$ and threshold radius $\rho_\alpha$.

\begin{table}[ht] \small 
\centering
\renewcommand{\arraystretch}{1.8}
\begin{tabular}{
|>{\raggedright\arraybackslash}m{0.29\textwidth}
|>{\centering\arraybackslash}m{0.12\textwidth}
|>{\centering\arraybackslash}m{0.18\textwidth}
|>{\centering\arraybackslash}m{0.32\textwidth}|
} \hline \rowcolor{teal!15}
\textbf{Example} & $\bs{\phi(s)}$ & $\bs{f(t)=\phi(\tau(t))}$ & $\bs{\rho_\alpha}$ \\\hline\hline
Triangle wave
    &$s$
    &$1-4\|t\|_{\T}$
    &$\displaystyle \frac{1-\alpha}{4}$\\
Cosine wave
    &$\displaystyle \sin\left(\frac{\pi s}{2}\right)$
    &$\cos(2\pi t)$
    &$\displaystyle \frac{\arccos(\alpha)}{2\pi}$\\
Cubic smoothstep wave
    &$\displaystyle \frac{3s-s^3}{2}$
    &$\displaystyle \frac{3\tau(t)-\tau(t)^3}{2}$
    &$\displaystyle \frac{1-\phi^{-1}(\alpha)}{4}$\\
Hyperbolic-tangent wave
    &$\displaystyle \frac{\tanh(cs)}{\tanh(c)}$
    &$\displaystyle
\frac{\tanh(c\tau(t))}{\tanh(c)}$
    &$\displaystyle \frac14- \frac{\operatorname{arctanh}(\alpha\tanh(c))}{4c} $\\[0.2cm]\hline
\end{tabular} \vspace{0.5cm}
\caption{\small Examples of antipodal waves obtained from the triangle wave
$\tau(t)=1-4\|t\|_{\T}$ by an odd increasing change of scale.}
\label{tab:antipodal-wave-examples}
\end{table}

The cosine is the motivating example in Table \ref{tab:antipodal-wave-examples}. In this case
$2\pi \rho_\alpha=\arccos\alpha $, so the reciprocal thresholds in Theorem \ref{thm:main} are precisely
$\alpha=\cos(2\pi/m)$ for integers $m\geq4$. In particular, for the original sign-correlation problem $\alpha=0$, one has $\rho_0=1/4$, and hence
\begin{align}
\lim_{n\to\infty}F_0(n)^{1/n}
&= \frac14.
\end{align}
Equivalently,
\begin{align}
F_0(n)
&= 4^{-n+o(n)}.
\end{align}
Thus $4^{-n}$ is the sharp exponential scale for the smallest possible proportion of points at which $n$ integer-frequency cosines have a common sign.

As observed in \cite{DGLLP}, the torus formulation of the cosine problem is closely related to the geometry underlying the Lonely Runner Conjecture. Integer frequencies may be viewed as runner velocities, while the threshold conditions prescribe arcs in which all runners are required to lie simultaneously. Related correlation questions for runners were studied by Perarnau and Serra \cite{PS16}.

\subsection{Outline}

For later use, we separate the two parts of $S_\alpha(A)$. For a finite nonempty set $A \subset \N$, write
\begin{align}
S_\alpha^+(A)
&:= \{t \in \T : f(at) > \alpha \text{ for every } a \in A\}, \\
S_\alpha^-(A)
&:= \{t \in \T : f(at) < -\alpha \text{ for every } a \in A\}.
\end{align}
Thus $S_\alpha(A) = S_\alpha^+(A) \cup S_\alpha^-(A)$. We use ordinary letters for scalar variables and bold letters for vectors. All sums and differences of subsets of $\T$ are taken in the group $\T$.

In Section \ref{sec:cyclic}, we prove Theorem \ref{thm:main-cover} using a base-$\rho_\alpha^{-1}$ expansion argument and a cube-packing estimate on the torus. Section \ref{sec:lower} establishes the universal lower bound by averaging translates of two antipodal boxes and applying a difference-set estimate on $\T$.

In Section \ref{sec:upper}, we prove the upper bounds using $\alpha$-isolating frequency sets. When $\rho_\alpha^{-1} \in \N$, we begin with $\{1,2,3\}$ and apply Theorem \ref{thm:main-cover} iteratively. When $\rho_\alpha^{-1} \notin \N$, a simpler dilation argument gives the desired estimate. In Section \ref{sec:rate}, we establish the existence of the exponential rate by proving a submultiplicative inequality for $F_\alpha$, using the asymptotic independence of widely separated frequency scales. Finally, Section \ref{sec:limitations} examines the polynomial loss in the reciprocal construction and discusses the remaining open questions.

\subsection*{Concurrent work}
A related preprint by Zang, Zhao, Zhao, and Zhou \cite{ZZZZ2026} was submitted to arXiv on October 23, 2026, and appeared in the arXiv announcement dated October 27, 2026. We became aware of their work through this announcement and are submitting the present manuscript on the same day. Our work was developed independently.

For the cosine sign-correlation problem, Theorem 1.3 of \cite{ZZZZ2026} gives the same lower bound $4^{1-n}$ and exponential decay rate $1/4$ as our Theorem \ref{thm:main}, together with an upper bound having the same polynomial exponent $\log_2(3/2)$. The proofs are also related: both lower bounds use an averaging argument and a difference-set inequality on the torus, while both upper bounds rely on the cyclic construction of Polak and Schrijver recorded in \cite{Polak2019}.

The two papers differ in scope. In \cite{ZZZZ2026}, the authors also establish sharp results for four frequencies and classify the cases of equality, whereas our work considers arbitrary thresholds for a class of symmetric periodic functions containing the cosine. We establish the existence of the exponential rate for every threshold, determine it when $\rho_\alpha^{-1} \in \N$, and obtain bounds in the remaining cases. Their upper-bound construction for general $n$ uses only odd frequencies, whereas ours includes even frequencies and eliminates the simultaneous negative-threshold set. We also examine the polynomial loss in the maximal iterative construction.

\subsection*{Acknowledgements} 
J.M. was partially supported by the Simons Foundation Grant $\# 453576$. The authors thank Lars Becker for some interesting conversations about this problem.
\subsection*{AI Disclosure} In preparing this manuscript, the authors used GPT-6 Astra to assist with language editing, syntax, and proofreading of author-written drafts, and to identify potentially relevant literature. The mathematical arguments and results are the authors' own. All AI-generated suggestions were reviewed by the authors, and all literature references identified with AI assistance were independently verified. The authors take full responsibility for the content and accuracy of the manuscript.

\section{The cyclic covering theorem}\label{sec:cyclic}

We begin with a packing estimate that will be used to prove the optimality assertion in Theorem \ref{thm:main-cover}.

\begin{lemma}\label{lem:packing}
Let $0 \leq \alpha < 1$ and suppose that $\rho_\alpha^{-1}\in\N$. If $\rho_\alpha < s \leq 1/2$, then the number of axis-parallel cubes of side length $s$ with disjoint interiors in $\T^d$ is at most $q_d$.
\end{lemma}

\begin{proof}
Set $\ell:=\rho_\alpha^{-1}$, and let $M_j(s)$ denote the largest number of such cubes in $\T^j$, with $M_0(s)=1$. A slice perpendicular to the first coordinate meets the interiors of at most $M_{j-1}(s)$ cubes. Integrating the number of intersections over the first coordinate gives $s M_j(s) \leq M_{j-1}(s).$
Since $s>1/\ell$, $M_j(s) < \ell M_{j-1}(s).$
The packing numbers are integers, and therefore $M_j(s) \leq \ell M_{j-1}(s) - 1.$ Iterating from $M_0(s)=1$ gives $M_j(s)\leq q_j$.
\end{proof}

\begin{proof}[Proof of Theorem \ref{thm:main-cover}]
Set $\ell=\rho_\alpha^{-1}$. The definition of $q_d$ becomes
\begin{align}
q_d & = \frac{(\ell - 2)\ell^d + 1}{\ell - 1},
\end{align}
and in particular $q_0=1$ and $q_d = \ell q_{d-1} - 1.$ Thus each $q_d$ is an integer and $q_d\equiv-1\pmod{\ell}$ for $d\geq1$. Moreover, for $d\geq2$ one has $q_{d-1}>\ell^{d-2}$, so $D_d$ consists of $d$ distinct positive integers.

We first prove the covering property. Fix $1\leq k<q_d$, write $q=q_d$, and suppose for contradiction that
\begin{align}
\left\|\frac{bk}{q}\right\|_{\T}
&\leq \frac1\ell
\end{align}
for every $b\in D_d$. Let $r$ be the unique integer with $1\leq r<q$ such that $r \equiv \ell^{-1}k \pmod q.$ Since $q=\ell q_{d-1}-1$, we have $q_{d-1}\equiv\ell^{-1}\pmod q$. Hence the elements of $D_d$ are congruent modulo $q$ to $\ell^{-1},1,\ell,\ldots,\ell^{d-2},$
and therefore
\begin{align}
\left\{\ell^j\frac rq\right\}
    \in \left[0,\frac1\ell\right]
    \cup \left[1-\frac1\ell,1\right)
&& 0\leq j\leq d-1.
\end{align}
The two nonzero endpoints cannot occur because $q\equiv-1\pmod{\ell}$. Thus the first $d$ digits in the base-$\ell$ expansion of $r/q$ all belong to $\{0,\ell-1\}$. Since $\ell\equiv1\pmod{\ell-1}$, it follows that
\begin{align}
\ell - 1 &\mid \left\lfloor \ell^d\frac rq\right\rfloor.
\end{align}

On the other hand, $(\ell - 1)q = (\ell - 2)\ell^d + 1,$ so
\begin{align}
\ell^d\frac rq
&= \frac{(\ell - 1)r}{\ell - 2}
    - \frac{r}{(\ell - 2)q}.
\end{align}
Write $r=(\ell-2)u+s$ with $0\leq s\leq\ell-3$. If $s=0$, then
\begin{align}
\left\lfloor \ell^d\frac rq\right\rfloor
&= (\ell - 1)u - 1,
\end{align}
while if $1\leq s\leq\ell-3$, then
\begin{align}
\left\lfloor \ell^d\frac rq\right\rfloor
&= (\ell - 1)u + s.
\end{align}
In either case this integer is not divisible by $\ell-1$, a contradiction. Hence, for every $1\leq k<q_d$, there is $b\in D_d$ such that
\begin{align}
\left\|\frac{bk}{q_d}\right\|_{\T}
&> \rho_\alpha.
\end{align}

It remains to prove optimality. Suppose that $q\geq2$ and $b_1,\ldots,b_d\in\Z$ satisfy
\begin{align}
\max_{1\leq j\leq d}
\left\|\frac{k b_j}{q}\right\|_{\T}
&> \rho_\alpha
\qquad\text{for every }1\leq k<q.
\end{align}
For $0\leq k<q$, define
\begin{align}
\bs z_k
&:= \left(\frac{k b_1}{q},\ldots,\frac{k b_d}{q}\right)
\in\T^d.
\end{align}
The points $\bs z_0,\ldots,\bs z_{q-1}$ are distinct. Moreover, the difference of any two distinct points has some coordinate at distance greater than $\rho_\alpha$ from zero. Since there are finitely many such differences, their minimum distance in the toroidal supremum norm is some $s>\rho_\alpha$, with $s\leq1/2$.

The axis-parallel cubes of side length $s$ centered at $\bs z_0,\ldots,\bs z_{q-1}$ therefore have disjoint interiors. Lemma \ref{lem:packing} gives $q\leq q_d$, proving the optimality assertion.
\end{proof}

\section{The universal lower bound}\label{sec:lower}

\begin{proposition}\label{prop:lower}
Let $0 \leq \alpha < 1$. For every integer $n \geq 1$ and every $n$-element set $A \subset \N$,
\begin{align}
P_\alpha(A) & \geq 4 \rho_\alpha^n.
\end{align}
Consequently, $F_\alpha(n) \geq 4 \rho_\alpha^n.$
\end{proposition}

\begin{proof}
Write $A = \{a_1,\ldots,a_n\}$ and consider the homomorphism $\Phi:\T \to \T^n,$ with $\Phi(t) = (a_1 t,\ldots,a_n t).$
Let
\begin{align}
C &:= (0,\rho_\alpha)^n, & 
\bs h & := \left(\frac12,\ldots,\frac12\right).
\end{align}
Since $\rho_\alpha \leq 1/4$, the sets $C$ and $C + \bs h$ are disjoint. For $\bs v \in \T^n$, define
\begin{align}
E_{\bs{v}} &:= \Phi^{-1}\bigl((\bs v + C)\cup(\bs v + C + \bs h)\bigr).
\end{align}
For each fixed $t \in \T$, the set of $\bs v \in \T^n$ for which $\Phi(t) \in \bs v + C$ has measure $\rho_\alpha^n$, and the same is true with $C + \bs h$ in place of $C$. Since the corresponding sets of translations are disjoint, Fubini's theorem gives
\begin{align}
\int_{\T^n} m(E_{\bs{v}})\,d\bs v &= 2 \rho_\alpha^n.
\end{align}
Hence there exists $\bs v \in \T^n$ such that $m(E_{\bs{v}}) \geq 2 \rho_\alpha^n.$ Fix such a $\bs v$ and write $E = E_{\bs{v}}$. Notice that $E$ is a finite union of open arcs.

We claim that $E - E \subset S_\alpha(A)$. Indeed, let $x,y \in E$. If $\Phi(x)$ and $\Phi(y)$ belong to the same one of the two boxes $\bs v + C$ and $\bs v + C + \bs h$, then, for every $1 \leq i \leq n$,
\begin{align}
\|a_i(x - y)\|_{\T} &< \rho_\alpha.
\end{align}
Thus $x - y \in S_\alpha^+(A)$. If $\Phi(x)$ and $\Phi(y)$ belong to opposite boxes, then, for every $1 \leq i \leq n$,
\begin{align}
\left\|a_i(x - y) - \frac12\right\|_{\T} &< \rho_\alpha,
\end{align}
and therefore $x - y \in S_\alpha^-(A)$.

The classical sumset inequality on the torus \cite{Macbeath1953} gives, for every finite union of open arcs $E \subset \T$, $m(E - E) \geq \min\{1,2m(E)\}$.
To see this, put $K = \overline E$ and $D = K - K$. Since $E$ is a finite union of open arcs, $m(D) = m(E - E)$. If $D = \T$, the assertion is immediate. Otherwise, choose $u \in \partial D$ and write $u = a - b$ with $a,b \in K$. Then $E - a$ and $E - b$ are disjoint subsets of $D$: otherwise, if $z \in (E - a)\cap(E - b)$, then
\begin{align}
u &= (z + a) - (z + b) \in E - E,
\end{align}
which contradicts $u \in \partial D$ because $E - E$ is open. Hence
\begin{align}
m(E - E) &= m(D) \\
&\geq m(E - a) + m(E - b) \\
&= 2m(E).
\end{align}

Combining this estimate with the inclusion $E - E \subset S_\alpha(A)$ gives
\begin{align}
P_\alpha(A) &= m(S_\alpha(A)) \\
&\geq m(E - E) \\
&\geq \min\{1,2m(E)\} \\
&\geq 4 \rho_\alpha^n,
\end{align}
where the last inequality follows from $m(E) \geq 2\rho_\alpha^n$ and $4\rho_\alpha^n \leq 1$. Taking the infimum over all $n$-element sets $A \subset \N$ gives the desired bound for $F_\alpha(n)$.
\end{proof}

\begin{remark}
The proof remains valid for a list of nonzero integer frequencies, even with repetitions. The two factors of $2$ in the lower bound have separate origins: the first comes from averaging a pair of antipodal boxes, while the second comes from the difference-set estimate. The argument concerns the un-shifted family $f(at)$; no common phase is implicit.
\end{remark}

\section{The upper bounds}\label{sec:upper}

Both upper bounds use the following notion.

\begin{definition}
Let $0 \leq \alpha < 1$. A finite set $B\subset\N$ is called $\alpha$-isolating if $1\in B$ and, writing $N=\max B$,
\begin{align}
S_\alpha^+(B) = \{t\in\T:\|t\|_{\T}<\rho_\alpha/N\}, 
    \qquad \text{ and } \qquad 
S_\alpha^-(B) = \varnothing.
\end{align}
For such a set, $P_\alpha(B) = 2\rho_\alpha/N.$
\end{definition}

\begin{proposition}\label{prop:upper-integer}
Suppose that $0 \leq \alpha < 1$ and $\rho_\alpha^{-1}\in\N$. Then there is a constant $C_\alpha > 0$ such that for every integer $n \geq 1$,
\begin{align}
F_\alpha(n)
&\leq C_\alpha
n^{\log_2((1-\rho_\alpha)/(1-2\rho_\alpha))}
\rho_\alpha^n.
\end{align}
\end{proposition}

\begin{proof}
The set $B_0=\{1,2,3\}$ is $\alpha$-isolating. Indeed, by the threshold-set description, if $t\in S_\alpha^+(B_0)$ then $\|t\|_{\T}<\rho_\alpha$ and $\|3t\|_{\T}<\rho_\alpha$. Choosing the representative with $|t|<\rho_\alpha$, the inequality $\rho_\alpha\leq1/4$ gives $|3t|<1-\rho_\alpha$, and therefore $|t|<\rho_\alpha/3$. Conversely, if $\|t\|_{\T}<\rho_\alpha/3$, then $\|jt\|_{\T}<\rho_\alpha$ for $j=1,2,3$. Thus $S_\alpha^+(B_0)=\{t\in\T:\|t\|_{\T}<\rho_\alpha/3\}$. For the negative set, write $t=1/2+s$ with $|s|<\rho_\alpha$. If the frequencies $1$ and $3$ are both below $-\alpha$, then $\|3s\|_{\T}<\rho_\alpha$, and the same argument gives $|s|<\rho_\alpha/3$. Hence $\|2t\|_{\T}=\|2s\|_{\T}<\rho_\alpha$, so the frequency $2$ is above $\alpha$. Therefore $S_\alpha^-(B_0)=\varnothing$.

We next establish the iterative step. Let $B$ be $\alpha$-isolating, let $N=\max B$, and suppose that $d\geq1$ satisfies $q_{d-1}\leq N$. Since $\max D_d=q_{d-1}<q_d=\min(q_dB)$, the two blocks are disjoint, and hence $|D_d\cup q_dB|=|B|+d$ and $\max(D_d\cup q_dB)=q_dN$.

Suppose first that $t\in S_\alpha^+(q_dB)$. Since $B$ is $\alpha$-isolating, there are $k\in\Z$ and $u\in\R$ so that
\begin{align}
t
&= \frac{k+u}{q_d},
\qquad
|u|<\frac{\rho_\alpha}{N}.
\end{align}
Assume that $k\not\equiv0\pmod {q_d}$. By Theorem \ref{thm:main-cover}, there is $b\in D_d$ such that
\begin{align}
\left\|\frac{bk}{q_d}\right\|_{\T}
&> \rho_\alpha.
\end{align}
Since $q_d=\ell q_{d-1}-1$ with $\ell=\rho_\alpha^{-1}$, the point $bk/q_d$ lies at distance at least $\rho_\alpha/q_d$ from the boundary of the interval $\{x\in\T:\|x\|_{\T}<\rho_\alpha\}$. On the other hand,
\begin{align}
\left|\frac{bu}{q_d}\right|
&< \frac{b\rho_\alpha}{Nq_d}
\leq \frac{\rho_\alpha}{q_d}.
\end{align}
Thus $\|bt\|_{\T}\geq\rho_\alpha$, so $t\notin S_\alpha^+(D_d\cup q_dB)$.

If $k\equiv0\pmod {q_d}$, then modulo $1$ we may write $t=u/q_d$. For every $b\in D_d$,
\begin{align}
|bt|
&< \frac{b\rho_\alpha}{Nq_d}
\leq \frac{\rho_\alpha}{q_d}
< \rho_\alpha.
\end{align}
Hence all the frequencies in $D_d$ are positive at $t$. It follows that
\begin{align}
S_\alpha^+(D_d\cup q_dB)
&= \left\{t\in\T:
\|t\|_{\T}<\frac{\rho_\alpha}{q_dN}\right\}.
\end{align}
Furthermore, $S_\alpha^-(D_d\cup q_dB) \subset S_\alpha^-(q_dB) = \varnothing,$ because $t\in S_\alpha^-(q_dB)$ would imply $q_dt\in S_\alpha^-(B)$. Thus $D_d\cup q_dB$ is again $\alpha$-isolating.

The preceding step also gives
\begin{align}
\rho_\alpha^{-(|B|+d)}P_\alpha(D_d\cup q_dB)
&= \rho_\alpha^{-|B|}P_\alpha(B)
    \frac{\ell^d}{q_d} \\
&= \rho_\alpha^{-|B|}P_\alpha(B)
    \frac{\ell-1}{\ell-2+\ell^{-d}} \\
&< \rho_\alpha^{-|B|}P_\alpha(B)
    \frac{\ell-1}{\ell-2}.
\end{align}
We now iterate using the largest admissible value of $d$. Starting from $B_0=\{1,2,3\}$, suppose that $B_j$ has been constructed, and write $N_j:=\max B_j$. Choose $d_j$ maximally subject to $q_{d_j-1}\leq N_j$, and set
\begin{align}
B_{j+1}
&:= D_{d_j}\cup q_{d_j}B_j.
\end{align}
By maximality, $q_{d_j-1} \leq N_j < q_{d_j}$, and so $N_{j+1} = N_jq_{d_j} > N_j^2.$ Thus $\log_\ell N_j \geq 2^j\log_\ell3.$ Also $\log_\ell N_j<|B_j|$ for every $j$: this holds initially and is preserved because $q_d<\ell^d$.

Fix $n\geq3$ and choose $j$ such that $|B_j|\leq n<|B_{j+1}|$. If $n=|B_j|$, take $A=B_j$. Otherwise, $1\leq n-|B_j|<d_j$, and
\begin{align}
q_{n-|B_j|-1}
&\leq q_{d_j-1}
\leq N_j,
\end{align}
so the same construction with $D_{n-|B_j|}$ gives an $\alpha$-isolating set $A$ with exactly $n$ elements. The number of construction steps used to obtain $A$ is at most $j+1$. Since $2^j\log_\ell3 \leq \log_\ell N_j < |B_j| \leq n,$ we have
\begin{align}
\left(\frac{\ell-1}{\ell-2}\right)^{j+1}
&\leq C_\alpha
n^{\log_2((\ell-1)/(\ell-2))}.
\end{align}

Finally, $\rho_\alpha^{-3}P_\alpha(B_0) = 2\ell^2/3.$ Repeated application of the preceding estimate therefore gives
\begin{align}
\rho_\alpha^{-n}P_\alpha(A)
&\leq \frac{2\ell^2}{3}
\left(\frac{\ell-1}{\ell-2}\right)^{j+1}
\\
&\leq C_\alpha
n^{\log_2((\ell-1)/(\ell-2))}
\\
&= C_\alpha
n^{\log_2((1-\rho_\alpha)/(1-2\rho_\alpha))}.
\end{align}
Multiplying by $\rho_\alpha^n$ gives the desired estimate for $n\geq3$. An appropriate choice of $C_\alpha$ takes care of $n=1,2$.
\end{proof}

\begin{remark}
The use of an even frequency in the isolating construction is essential for breaking a symmetry that is present for odd frequencies. Indeed, if every $a\in A$ is odd, then periodicity and the antipodal symmetry of $f$ give
\begin{align}
f\left(a\left(t+\frac12\right)\right)
&= -f(at).
\end{align}
Consequently, translation by $1/2$ interchanges $S_\alpha^+(A)$ and $S_\alpha^-(A)$, and in particular
\begin{align}
m(S_\alpha^+(A))
&= m(S_\alpha^-(A)).
\end{align}
The frequency $2$ in the seed $B_0=\{1,2,3\}$ breaks this symmetry and allows the simultaneous negative set to be empty while the simultaneous positive set remains nonempty. This asymmetry is then preserved by the subsequent isolating construction.
\end{remark}

\begin{proposition}\label{prop:upper-noninteger}
Suppose that $0 \leq \alpha < 1$ and $\rho_\alpha^{-1}\notin\N$. Then there is a constant $C_\alpha > 0$ such that for every integer $n \geq 1$,
\begin{align}
F_\alpha(n)
&\leq C_\alpha
\left\lfloor\rho_\alpha^{-1}\right\rfloor^{-n}.
\end{align}
\end{proposition}

\begin{proof}
Set $\ell=\lfloor\rho_\alpha^{-1}\rfloor$. Then $\ell\geq4$ and $\rho_\alpha<1/\ell$. Let $B_0=\{1,\ldots,N\}$, where
\begin{align}
N &:= \max\left\{3,
    \left\lceil
    \frac{\rho_\alpha}{1-\ell\rho_\alpha}
    \right\rceil\right\}.
\end{align}

We first observe that $B_0$ is $\alpha$-isolating. Certainly, if $\|t\|_{\T}<\rho_\alpha/N$, then $\|jt\|_{\T}<\rho_\alpha$ for every $1\leq j\leq N$. Conversely, suppose that $t\in S_\alpha^+(B_0)$. The frequency $1$ gives $\|t\|_{\T}<\rho_\alpha$, so choose the representative satisfying $|t|<\rho_\alpha$. If $|t|\geq\rho_\alpha/N$, let $j=\lceil\rho_\alpha/|t|\rceil$. Then $1\leq j\leq N$ and
\begin{align}
\rho_\alpha
&\leq j|t|
< 2\rho_\alpha
\leq \frac12.
\end{align}
Hence $\|jt\|_{\T}\geq\rho_\alpha$, a contradiction. Therefore
\begin{align}
S_\alpha^+(B_0)
&= \left\{t\in\T:
\|t\|_{\T}<\frac{\rho_\alpha}{N}\right\}.
\end{align}
Moreover, $B_0$ contains $\{1,2,3\}$, whose frequencies have no common negative point, as observed in the preceding proof. Thus $S_\alpha^-(B_0)=\varnothing$.

Now let $B$ be $\alpha$-isolating with $\max B\geq N$, and set $\widetilde B := \{1\}\cup\ell B.$ If $t\in S_\alpha^+(\ell B)$, then there are $k\in\{0,\ldots,\ell-1\}$ and $u\in\R$ such that
\begin{align}
t
&= \frac{k+u}{\ell},
\qquad
|u|<\frac{\rho_\alpha}{\max B}.
\end{align}
If $k\neq0$, then
\begin{align}
\|t\|_{\T}
&> \frac1\ell-\frac{\rho_\alpha}{\ell\max B}
\geq \rho_\alpha,
\end{align}
where the last inequality follows from $\max B\geq N\geq\rho_\alpha/(1-\ell\rho_\alpha)$. Thus the frequency $1$ excludes every noncentral component.

If $k=0$, then
\begin{align}
\|t\|_{\T}
&< \frac{\rho_\alpha}{\ell\max B}
< \rho_\alpha,
\end{align}
so the frequency $1$ is positive there. It follows that
\begin{align}
S_\alpha^+(\widetilde B)
&= \left\{t\in\T:
\|t\|_{\T}<\frac{\rho_\alpha}{\ell\max B}\right\}.
\end{align}
Also, $S_\alpha^-(\widetilde B)\subset S_\alpha^-(\ell B)=\varnothing$. Hence $\widetilde B$ is again $\alpha$-isolating.

Starting from $B_0$, define recursively $B_{j+1} := \{1\}\cup\ell B_j.$ Then each $B_j$ is $\alpha$-isolating, $|B_j| = N+j,$ and $\max B_j = N\ell^j.$ Consequently,
\begin{align}
P_\alpha(B_j)
&= \frac{2\rho_\alpha}{N\ell^j}.
\end{align}
For every $n\geq N$, taking $j=n-N$ therefore gives
\begin{align}
F_\alpha(n)
&\leq \frac{2\rho_\alpha\ell^N}{N}\ell^{-n}.
\end{align}
An appropriate choice of $C_\alpha$ takes care of the finitely many  $1\leq n<N$, and so for all $n\geq1$,
\begin{align}
F_\alpha(n)
&\leq C_\alpha
\left\lfloor\rho_\alpha^{-1}\right\rfloor^{-n}.
\end{align}
\end{proof}

\section{The exponential rate}\label{sec:rate}

We complete the proof of Theorem \ref{thm:main} by establishing the following submultiplicative inequality.

\begin{proposition}\label{prop:submultiplicative}
Let $0 \leq \alpha < 1$. For all integers $n,m \geq 1$,
\begin{align}
F_\alpha(n+m)
&\leq F_\alpha(n)F_\alpha(m).
\end{align}
Consequently, the limit exists.
\end{proposition}

\begin{proof}
Fix $\varepsilon>0$, and choose sets $A,B\subset\N$ with $|A|=n$ and $|B|=m$ such that
\begin{align}
P_\alpha(A)
&\leq F_\alpha(n)+\varepsilon,\\
P_\alpha(B)
&\leq F_\alpha(m)+\varepsilon.
\end{align}
For an integer $K\geq1$, let $KB=\{Kb:b\in B\}$. If $K>\max A$, then $A$ and $KB$ are disjoint, so $|A\cup KB|=n+m$.

We use the elementary fact that if $E,F\subset\T$ are finite unions of open arcs, then
\begin{align}
m\bigl(E\cap\{t\in\T:Kt\in F\}\bigr)
&\longrightarrow m(E)m(F)
\end{align}
as $K\to\infty$. Indeed, it is enough to check this when $E$ and $F$ are arcs, in which case we can partition the circle into $K$ intervals of length $1/K$, and see that on each complete interval contained in $E$, the pre-image of $F$ occupies measure $m(F)/K$. Only the at most two intervals intersecting the endpoints of $E$ contribute an error of order $1/K$.

Since
\begin{align}
S_\alpha^+(A\cup KB)
&= S_\alpha^+(A)\cap\{t\in\T:Kt\in S_\alpha^+(B)\},
\\
S_\alpha^-(A\cup KB)
&= S_\alpha^-(A)\cap\{t\in\T:Kt\in S_\alpha^-(B)\},
\end{align}
we obtain
\begin{align}
P_\alpha(A\cup KB)
&\longrightarrow m(S_\alpha^+(A))m(S_\alpha^+(B))
        + m(S_\alpha^-(A))m(S_\alpha^-(B))\\
&\leq P_\alpha(A)P_\alpha(B).
\end{align}
Therefore, for all sufficiently large $K$,
\begin{align}
F_\alpha(n+m)
&\leq P_\alpha(A\cup KB) \\
&\leq P_\alpha(A)P_\alpha(B)+\varepsilon \\
&\leq \bigl(F_\alpha(n)+\varepsilon\bigr)
        \bigl(F_\alpha(m)+\varepsilon\bigr)
        +\varepsilon.
\end{align}
Letting $\varepsilon\to0$ gives $F_\alpha(n+m) \leq F_\alpha(n)F_\alpha(m).$ Moreover, by Proposition \ref{prop:lower}, $F_\alpha(n)>0$ for every $n$. Hence $\log F_\alpha(n)$ is subadditive, and Fekete's lemma shows that
\begin{align}
\lim_{n\to\infty}\frac{1}{n}\log F_\alpha(n)
&= \inf_{n\geq1}\frac{1}{n}\log F_\alpha(n).
\end{align}
Exponentiating proves that the limit exists.
\end{proof}

\begin{proof}[Proof of Theorem \ref{thm:main}]
Proposition \ref{prop:lower} gives the lower bound, while Propositions \ref{prop:upper-integer} and \ref{prop:upper-noninteger} give the upper bounds in parts \ref{pt:main-integer} and \ref{pt:main-noninteger}, respectively. Proposition \ref{prop:submultiplicative} shows that the limit of the exponential rate exists.

If $\rho_\alpha^{-1}\in\N$, taking $n$th roots in the lower and upper bounds gives $\rho_\alpha \leq \lim_{n\to\infty}F_\alpha(n)^{1/n} \leq \rho_\alpha,$ and hence $\lim_{n\to\infty}F_\alpha(n)^{1/n} = \rho_\alpha$. If $\rho_\alpha^{-1}\notin\N$, the same argument  produces the desired inequality.
\end{proof}

\section{The remaining loss and open questions}\label{sec:limitations}

\begin{proposition}\label{prop:loss}
Suppose that $0 \leq \alpha < 1$ and $\rho_\alpha^{-1}\in\N$, and let $B_j$ be the maximal sequence constructed in the proof of Proposition \ref{prop:upper-integer}. Then
\begin{align}
P_\alpha(B_j)
&\simeq 
|B_j|^{\log_2((1-\rho_\alpha)/(1-2\rho_\alpha))}
\rho_\alpha^{|B_j|}
\qquad (j\to\infty),
\end{align}
where the implicit constants depend only on $\alpha$.
\end{proposition}

\begin{proof}
Set $\ell=\rho_\alpha^{-1}$, $N_j=\max B_j$, and let $d_j$ denote the size of the $j$th construction step. By maximality, $q_{d_j-1} \leq N_j < q_{d_j}.$
Since $q_d=\ell q_{d-1}-1$, we have $N_j < q_{d_j} \leq \ell N_j-1.$ Thus, writing $L_j=\log_\ell N_j$, $2L_j < L_{j+1} < 2L_j+1,$ and consequently $L_j\simeq 2^j$. Moreover, $q_{d_j}>N_j$ and $q_{d_j}<\ell^{d_j}$ imply $d_j>L_j$.

The exact loss at the $j$th step is
\begin{align}
\frac{\ell^{d_j}}{q_{d_j}}
&= \frac{\ell-1}{\ell-2}
\left(1+\frac{\ell^{-d_j}}{\ell-2}\right)^{-1}.
\end{align}
Therefore
\begin{align}
\rho_\alpha^{-|B_j|}P_\alpha(B_j)
&= \frac{2\ell^2}{3}
\left(\frac{\ell-1}{\ell-2}\right)^j
\prod_{i=0}^{j-1}
\left(1+\frac{\ell^{-d_i}}{\ell-2}\right)^{-1}.
\end{align}
Since $d_i>L_i\gtrsim2^i$, the sum $\sum_i\ell^{-d_i}$ converges, and hence the product converges to a positive limit. It follows that
\begin{align}
\rho_\alpha^{-|B_j|}P_\alpha(B_j)
&\simeq 
\left(\frac{\ell-1}{\ell-2}\right)^j.
\end{align}

On the other hand,
\begin{align}
\log_\ell\left(\rho_\alpha^{-|B_j|}P_\alpha(B_j)\right)
&= |B_j|-L_j+\log_\ell\left(\frac{2}{\ell}\right),
\end{align}
so $|B_j|=L_j+O(j)$. Hence $|B_j|\simeq 2^j$, and therefore
\begin{align}
\left(\frac{\ell-1}{\ell-2}\right)^j
&\simeq 
|B_j|^{\log_2((\ell-1)/(\ell-2))}.
\end{align}
Since
$\log_2((\ell-1)/(\ell-2))
=
\log_2((1-\rho_\alpha)/(1-2\rho_\alpha))$,
the result follows.
\end{proof}

This leaves open whether the polynomial loss is a feature of the problem itself or only of the construction.

\begin{question}\label{q:constant}
Suppose that $0 \leq \alpha < 1$ and $\rho_\alpha^{-1}\in\N$. Is there a constant $C_\alpha>0$ such that $F_\alpha(n) \leq C_\alpha\rho_\alpha^n$ for every integer $n\geq1$?
\end{question}

Theorem \ref{thm:main-cover} also shows that a strict cyclic cover with $d$ multipliers cannot have modulus $q$ satisfying $q\rho_\alpha^d\to1$: the largest possible ratio tends to
\begin{align}
\frac{1-2\rho_\alpha}{1-\rho_\alpha}
&< 1.
\end{align}
Thus this particular modification of the covering step cannot remove the accumulated loss. Neither this restriction nor Proposition \ref{prop:loss}, however, gives a lower bound forcing $\rho_\alpha^{-n}F_\alpha(n)$ to diverge.

When $\rho_\alpha^{-1}\notin\N$, the lower and upper exponential bounds in Theorem \ref{thm:main} do not in general coincide. Determining the exact value of the limit in this case remains beyond the present argument.

\setlength{\parindent}{0pt}

\end{document}